\documentclass[11pt,letterpaper]{amsart}
\usepackage{fullpage}
\usepackage{amscd}
\usepackage{latexsym}
\usepackage{amssymb,amsmath}
\usepackage{amsrefs}
\usepackage[all]{xy}
\usepackage{color}

\DeclareFontFamily{U}{wncy}{}
\DeclareFontShape{U}{wncy}{m}{n}{<->wncyr10}{}
\DeclareSymbolFont{mcy}{U}{wncy}{m}{n}
\DeclareMathSymbol{\Sha}{\mathord}{mcy}{"58}
\DeclareMathOperator{\norm}{Norm}
\DeclareMathOperator{\Disc}{Disc}
\DeclareMathOperator{\Nm}{Nm}
\DeclareMathOperator{\Pic}{Pic}
\DeclareMathOperator{\divis}{div}
\DeclareMathOperator{\GL}{GL}
\DeclareMathOperator{\PGL}{PGL}

\DeclareMathOperator{\PSp}{PSp}
\DeclareMathOperator{\Hom}{Hom}
\DeclareMathOperator{\Sel}{Sel}

\newcommand{\sep}{{\operatorname{sep}}}

\newcommand{\Z}{{\mathbb Z}}
\newcommand{\ZZ}{{\mathbb Z}}
\newcommand{\Q}{{\mathbb Q}}
\newcommand{\QQ}{{\mathbb Q}}
\newcommand{\FF}{{\mathbb F}}
\newcommand{\GG}{{\mathbb G}}
\newcommand{\RR}{{\mathbb R}}
\newcommand{\PP}{{\mathbb P}}

\newcommand{\sC}{\mathcal{C}}

\newcommand{\sO}{\mathcal{O}}
\newcommand{\Crst}{\sC_{rst}}
\newcommand{\Jrst}{\sJ_{rst}}
\newcommand{\tCrst}{\widetilde{\sC}_{rst}}
\newcommand{\tJrst}{\widetilde{\sJ}_{rst}}
\newcommand{\txi}{\tilde{\xi}}
\newcommand{\tPhi}{\tilde{\Phi}}
\newcommand{\tPsi}{\tilde{\Psi}}

\newcommand{\sJ}{\mathcal{J}}
\newcommand{\krn}{\Sigma}

\newcommand{\C}{C}
\newcommand{\J}{J}
\newcommand{\tJ}{{\widetilde\J}}
\newcommand{\tC}{\widetilde{\mathcal{\C}}}

\newcommand{\tA}{{\widetilde{A}}}
\newcommand{\tG}{{\widetilde{G}}}
\newcommand{\tH}{{\widetilde{H}}}
\newcommand{\tT}{{\widetilde{T}}}
\newcommand{\tlambda}{{\tilde{\lambda}}}
\newcommand{\tq}{{\tilde{q}}}
\newcommand{\tx}{{\tilde{x}}}

\newcommand{\K}{\mathcal{K}}
\newcommand{\tK}{\tilde{\K}}

\newcommand{\F}{\mathcal{F}}

\newcommand{\sA}{\mathcal{A}}

\newcommand{\tphi}{\tilde{\phi}}
\newcommand{\tgamma}{\tilde{\gamma}}

\newcommand{\JacC}{{\hbox{Jac}_{\lower.5pt\hbox{$_\sC$}}}}
\newcommand{\JacF}{{\hbox{Jac}_{\lower.5pt\hbox{$_\F$}}}}

\newcommand{\Gal}{\mathrm{Gal}}
\newcommand{\Res}{\mathrm{Res}}
\newcommand{\res}{\mathrm{res}}

\newcommand{\kbar}{\overline{k}}
\newtheorem{thm}{Theorem}

\newtheorem{lemma}[thm]{Lemma}

\newtheorem{cor}[thm]{Corollary}
\theoremstyle{definition}
\newtheorem{example}[thm]{Example}
\newtheorem{rmk}[thm]{Remark}
 
\theoremstyle{definition}

\theoremstyle{remark}

\newcommand{\adj}[1]{#1^\text{adj}}

\title{Descent via $(3,3)$-isogeny on Jacobians of genus $2$ curves}

\author{N.\ Bruin}
\address{Department of Mathematics, Simon Fraser University,
Burnaby, BC, CANADA, V5A 1S6}
\email{nbruin@cecm.sfu.ca}
\author{E.V.\ Flynn}
\address{Mathematical Institute, University of Oxford, 24--29 St.\ Giles,
Oxford OX1 3LB, United Kingdom}
\email{flynn@maths.ox.ac.uk}
\author{D.\ Testa}
\address{Mathematics Institute, Zeeman Building, University of Warwick, 
Coventry CV4 7AL, United Kingdom}
\email{D.Testa@warwick.ac.uk}

\thanks{The first author is partially supported
by an NSERC grant. The second and third authors are partially
supported by EPSRC grant EP/F060661/1.}
\date{1 January, 2014}
\begin{document}

\begin{abstract}
We give parametrisation of curves~$\C$ of genus~2
with a maximal isotropic $(\ZZ/3)^2$ in $\J[3]$,
where $\J$ is the Jacobian variety of~$\C$, and develop the 
theory required to perform descent via $(3,3)$-isogeny.
We apply this to several examples,
where it can shown that non-reducible Jacobians have
nontrivial $3$-part of the Tate-Shafarevich group. 
\end{abstract}

\maketitle

\section{Introduction}\label{sec:intro}

In this article we consider curves $\C$ of genus~$2$ over a field~$k$ of characteristic not~2 or~3, that have special structure in the $3$-torsion of their Jacobians $\J$. In particular, we consider the situation where $\J[3](k)$ contains a group $\Sigma(k)$ of order $9$. As we show in Section~\ref{sec:threetors}, such a curve $\C$ can be given by a model of the form
\begin{equation}\label{eq:genus2tors3}
y^2 = F(x) = G(x)^2 + \lambda H(x)^3,
\end{equation}
where $G(x)$ is cubic and $H(x)$ is quadratic in~$x$. The divisor
$D = [ (x, G(x)) + (x', G(x')) - \kappa ]$, where $x,x'$ are the roots of~$H(x)$
and~$\kappa$ is a canonical divisor, represents a point in $\J[3](k)$, as can be
seen from the fact that $3D$ is linearly equivalent to the divisor of the
function $y-G(x)$.

The curves of interest to us can be expressed as such a model in
\emph{several} ways. As we show in Lemma~\ref{L:trivial_pairing}, the Weil
pairing of the $3$-torsion points can be easily expressed in terms of the
corresponding polynomials $G(x)$ and $H(x)$. This allows us to fully describe
the genus $2$ curves that have a subgroup $\Sigma(k)\subset \J[3](k)$ of size
$9$ on which the Weil pairing is trivial.

In Section~\ref{sec:level} we phrase the question of classification of such
curves in terms of \emph{partial level structures} on principally polarized
abelian surfaces. The relevant moduli space is
$\sA_2(\Sigma)$, with $\Sigma=\ZZ/3\times\ZZ/3$. In Section~\ref{sec:param} we
determine a genus $2$ curve $\Crst$ over $k(r,s,t)$ such that a sufficiently
general curve of the type we are interested in, can be obtained by specializing
$r,s,t$. Our construction identifies $k(r,s,t)$ with the function field of the
moduli space $\sA_2(\Sigma)$, thereby giving a particularly explicit proof of
its rationality. Furthermore, we observe that if $\J$ is equipped with a
$\Sigma$-level structure, then the quadratic twist $\J^{(d)}$ is naturally
equipped with a $\Sigma^{(d)}$ level structure. Thus we find that
$\sA_2(\Sigma)$ and $\sA_2(\Sigma^{(d)})$ are naturally isomorphic.

Since $\Sigma\subset \J[3]$ is maximally isotropic with respect to the Weil pairing, we have that $\tJ=\J/\Sigma$ is also principally polarized and in fact has a $\Sigma^\vee$-level structure. In Section~\ref{sec:isog} we identify a curve $\tCrst$ such that its Jacobian $\tJrst$ is $\Jrst/\Sigma$. In fact, we observe that $\Sigma^\vee=\Sigma^{(-3)}$ and hence that the quadratic twist $\tCrst^{(-3)}\simeq \sC_{r's't'}$, where $(r',s',t')=\psi_0(r,s,t)$ is a birational transformation that we explicitly determine. While the final formulas we find are quite manageable, the proof that they are correct requires some significant computation.

In Theorem~\ref{thm:Amtx} and Corollary~\ref{cor:magic} we observe some
remarkable relations between the models for $\Crst$ and $\tCrst$. We also
identify the natural action of $\PGL_2(\FF_3)$ on $\sA_2(\Sigma)$, as well as
the involution $\Crst\to \tCrst^{(-3)}$, as automorphisms of $k(r,s,t)$.

In Section~\ref{sec:descent} we get to the original motivation of this paper. If
we take $\J$ to be a Jacobian of a genus $2$ curve with a $\Sigma$ level
structure over a number field $k$, then it is particularly easy to compute
interesting information about $\J(k)$ and $\Sha(\J/k)[3]$ via $(3,3)$-isogeny
descent, using ideas from \cite{schaefer:sel} and \cite{BPS:quartic}. This
allows us to give examples of various absolutely simple abelian surfaces with
interesting structures in $\Sha(\J/k)[3]$, see Examples~\ref{E:-3,-3,-3},
\ref{E:-2,1,2}, \ref{E:2,-1,-2}.


\section{Three-torsion on genus two Jacobians}\label{sec:threetors}

Let $k$ be a field of characteristic different from $2,3$ and let  $\C$ 
be a smooth projective curve of genus~$2$ over~$k$ given by an affine model
\[ \C\colon y^2=F(x),\]
where $F(x)\in k[x]$ is of degree $6$ (this only (mildly) restricts 
the admissible $\C$ if $k$ is a finite field of~$5$ elements). Let $\J$ 
be the Jacobian of $\C$.
The group $\J(k)$ is isomorphic to the group 
of divisor classes of $k$-rational degree~$0$ divisors on~$\C$. Since $\C$ is of genus $2$, 
every degree~$0$ class contains a representative of the form
\[ D-\kappa,\]
where $D$ is an effective divisor of degree $2$ and $\kappa$ is an 
effective canonical divisor. Furthermore, for any non-principal class, 
the divisor $D$ is unique. For $\kappa$ we have choice, since effective 
canonical divisors are exactly the fibers of the hyperelliptic 
double cover $x\colon \C \to \PP^1$. We write $\iota\colon \C\to \C$ 
for the hyperelliptic involution, i.e., $\iota(x,y)=(x,-y)$.

\begin{lemma}\label{L:disjoint}
Let $\krn\subset \Pic(\C/k)[3]$ be a subgroup of size $9$. 
Then $\krn$ can be generated by a pair of divisors
  \[D_1-\kappa_1,D_2-\kappa_2\]
where $D_1,D_2$ are effective divisors of degree $2$ and $\kappa_1,\kappa_2$ 
are effective canonical divisors, with the supports of 
$D_1,D_2,\kappa_1,\kappa_2$ pairwise disjoint. A fortiori, we can 
ensure that $x_*(D_1)$ and $x_*(D_2)$ are disjoint.
\end{lemma}
\begin{proof}
First, note that we can choose $\kappa_1$ and $\kappa_2$ to be any fiber of
$x\colon \sC\to\PP^1$ over points in $\PP^1(k)$. Since $k$ has characteristic
$0$ or at least $5$, we know that $\#\PP^1(k)\geq 6$. Since $x_*(D_1+D_2)$ is
supported on at most $4$ rational points, we can choose $\kappa_1,\kappa_2$ with
disjoint support from $D_1,D_2$.

It remains to show that we can choose $D_1,D_2$ with disjoint support. Since
these divisors are uniquely determined, we lose no generality 
by assuming that $k$ is algebraically closed. Therefore, we assume 
that $\krn$ is generated by the classes
\[ T_1=[P_1+Q_1-\kappa_1] \text{ and } T_2=[P_2+Q_2-\kappa_2],\]
where $P_1,Q_1,P_2,Q_2\in \C(k)$. We want to ensure that $\{P_1,\iota
P_1,Q_1,\iota Q_1\}$ 
and $\{P_2,\iota P_2,Q_2,\iota Q_2\}$ are disjoint. If they are not, 
we can assume that $P_1=P_2$ and it follows that
\[T_3=T_1-T_2=[Q_1+\iota Q_2 -\kappa].\]
A straightforward computation shows that if $T_4=T_1+T_2=[P_4+Q_4-\kappa]$ 
has $P_4=P_1$ or $P_4=\iota P_1$ then
either $T_4=[2P_1-\kappa]$ and $T_3=[2Q_1-\kappa]$, so that 
choosing $T_3,T_4$ works, or $\langle T_1,T_2\rangle$ is not of size $9$.
Therefore, one of the choices $(T_1,T_2)$ or $(T_1,T_4)$ or $(T_3,T_4)$ 
satisfies our criteria.
\end{proof}

\begin{example}\label{E:special_curve}
At this point it may be worth noting that rather exceptional configurations 
of support for $3$-torsion do occur. For instance, for
\[\C\colon y^2=x^6+r x^3 +1\]
we see that $T_1=[(0,1)+\infty^+-\kappa]$ and $T_2=[(0,1)+\infty^--\kappa]$ 
are $3$-torsion points. In fact, it is straightforward to check 
that any genus 2 Jacobian with two independent $3$-torsion points 
such that the group generated by those $3$-torsion points is supported 
on only $4$ points of the curve must be isomorphic to a curve of this form.
\end{example}

\begin{lemma}\label{L:general_form} Let $\C$ be a curve of genus $2$. 
Then $\Pic(\C/k)[3]$ has a subgroup $\krn$ of size $9$ if and only 
if $\C$ admits a model of the form
\[ \C\colon y^2=F(x)=G_1(x)^2+\lambda_1 H_1(x)^3 = G_2(x)^2+\lambda_1 H_2(x)^3,\]
where $H_1,H_2,G_1,G_2,F\in k[x]$ and $\lambda_1,\lambda_2\in k^\times$ 
and $H_1,H_2$ are of degree $2$ and $\gcd(H_1,H_2)=1$.
\end{lemma}
\begin{proof}
Suppose that $T\in\Pic(\C/k)[3]$ is non-trivial. We assume 
that $T=[D-\kappa_\infty]$, where $\kappa_\infty$ is the fiber 
above $x=\infty$ and $D$ is an effective divisor with support disjoint 
from $\kappa_\infty$.
Then $x_*(D)$ can be described by $H(x)=0$, where $H(x)\in k[x]$ is a 
quadratic monic polynomial. Since $3T$ is the principal class, there 
is a function $g\in k(\C)$ such that
\[\divis(g)= 3D-3\kappa_\infty\]
and it is straightforward to check that we must have $g=y-G(x)$ for 
some $G(x)\in k[x]$, with $\deg(G)\leq 3$. It follows that
\[y^2=F(x)=G(x)^2+\lambda H(x)^3,\]
and conversely, that any such decomposition of $F(x)$ gives rise to 
a $3$-torsion point $T=[D-\kappa]$, where $D$ is the effective 
degree $2$ divisor described by the vanishing of $\{y-G(x),H(x)\}$. 
The class $2T=-T$ is then described by the vanishing of $\{y+G(x),H(x)\}$.

The existence of $\krn$ as stated in the lemma would lead to $4$ 
decompositions of the type described and simple combinatorics shows 
that not all quadratic polynomials $H(x)$ featuring in them can be equal. 
This proves the lemma.
\end{proof}

The torsion subgroup scheme $\J[3]$ comes equipped with a non-degenerate, 
bilinear, alternating \emph{Weil pairing}
\[e_3\colon \J[3]\times \J[3] \to \mu_3,\]
where $\mu_3$ is the group scheme representing the cube roots of unity.

We say that a subgroup $\krn\subset \J[3]$ is \emph{isotropic} if $e_3$ restricts 
to the trivial pairing on $\krn$. If $\krn$ is of degree $9$ then $\krn$ 
is \emph{maximal isotropic}, meaning $\krn$ is not properly contained 
in an isotropic subgroup. The nondegeneracy of $e_3$ then induces an 
isomorphism $\J[3]/\krn\to \krn^\vee=\Hom(\krn,\mu_3)$. In fact, we have 
a direct sum decomposition $\J[3]=\krn\times \krn^\vee$.

In particular, if $\krn=\Z/3\times\Z/3$ then we have 
$\krn^\vee=\mu_3\times\mu_3$.

\begin{lemma} Let
\[ C\colon y^2=F(x)=G_1(x)^2+\lambda_1 H_1(x)^3 
= G_2(x)^2+\lambda_2 H_2(x)^3\]
be a genus $2$ curve with $H_1,H_2$ quadratic monic polynomials 
and $H_1\neq H_2$. For $i \in \{1,2\}$, let $D_i=\{y-G_i(x),H_i(x)\}$ and 
let $T_i=[D_i-\kappa]\in \Pic(\C/k)[3]$. Then
\[e_3(T_1,T_2)=\frac{\lambda_2}{\lambda_1}
\frac{\Res(G_2-G_1,H_2)}{\Res(G_1-G_2,H_1)}\]
\end{lemma}
\begin{proof}
We choose canonical divisors $\kappa_1$ and $\kappa_2$ above $x=r_1$ 
and $x=r_2$ respectively, such that the divisors
\[D_1-\kappa_1\text{ and } D_2-\kappa_2\]
have disjoint support. We have the functions
\[g_i=\frac{y-G_i}{(x-r_i)^3} \text{ with }\divis(g_i)=3D_i-3\kappa_i.\]
We can compute the pairing via
\[e_3(T_1,T_2)=\frac{g_1(D_2-\kappa_2)}{g_2(D_1-\kappa_1)},\]
where evaluating a function on a divisor is defined 
as $g(\sum n_P P)=\prod g(P)^{n_P}$.
Evaluating $y-G_1(x)$ at $D_2$ means evaluating $G_2(x)-G_1(x)$ at the roots 
of $H_2(x)$, yielding $\Res(G_2-G_1,H_2)$. Evaluating $(x-r_1)$ at $D_2$ yields
$H_2(r_1)$. Noting that $\kappa_2=(r_2,\sqrt{F(r_2)})+(r_2,-\sqrt{F(r_2)})$, 
we see that
\[g_1(\kappa_2)=\frac{G_1(r_2)^2-F(r_2)}{(r_2-r_1)^6}
=\lambda_1\frac{-H_1(r_2)^3}{(r_2-r_1)^6}\]
and hence that
\[g_1(D_2-\kappa_2)=\frac{\Res(G_2-G_1,H_2)(r_2-r_1)^6}
{-\lambda_1H_1(r_2)^3H_2(r_1)^3}.\]
Symmetry yields the result stated in the lemma.
\end{proof}

We can characterize when $e_3(T_1,T_2)=1$ in terms of the polynomials
$G_i,H_i$ 
in the following way. First note that
\[G_2^2-G_1^2=\lambda_1 H_1^3-\lambda_2 H_2^3.\]
Writing $\alpha_1,\alpha_2,\alpha_3$ for the cube roots 
of $\lambda_2/\lambda_1$ in an algebraic closure of $k$, we find that
\[(G_2-G_1)(G_2+G_1)
=\lambda_1(H_1-\alpha_1 H_2)(H_1-\alpha_2 H_2)(H_1-\alpha_3 H_2).\]
It follows that the roots of the quadratic polynomials $H_1-\alpha_i H_2$ 
are the same as the roots of $G_2-G_1$ and $G_2+G_1$. The way in which 
they distribute determines the Weil pairing.

\begin{lemma}\label{L:trivial_pairing}
The pairing $e_3(T_1,T_2)=1$ if and only if none of the polynomials 
$(H_1-\alpha_i H_2)$ divide $G_2-G_1$.
\end{lemma}
\begin{proof}
First suppose that $G_2-G_1=L_1 (H_1-\alpha_1 H_2)$. Then
\[e_3(T_1,T_2)
=\frac{\lambda_2}{\lambda_1}\frac{\Res(G_2-G_1,H_2)}{\Res(G_1-G_2,H_1)}
=\frac{\lambda_2}{\lambda_1}\frac{\Res((H_1-\alpha_1H_2)L_1,H_2)}
{\Res((H_1-\alpha_1H_2)L_1,H_1)}=\frac{\lambda_2}{\lambda_1\alpha_1^2}
\frac{\Res(L_1,H_2)}{\Res(L_1,H_1)}.
\]
Observe that $L_1$ must divide one of the other factors $(H_1-\alpha_j H_2)$, 
say for $j=2$. Therefore, $\Res(L_1,H_1)=\Res(L_1,\alpha_2H_2)$. 
Since $\deg(L_1)=1$ and $\lambda_2/\lambda_1=\alpha_1^3$ we obtain
\[e_3(T_1,T_2)=\alpha_1\frac{\Res(L_1,H_2)}{\Res(L_1,H_1)}
=\alpha_1\frac{\Res(L_1,H_2)}{\Res(L_1,\alpha_2H_2)}
=\frac{\alpha_1}{\alpha_2},\]
which is indeed a primitive cube root of unity.

In the remaining situation we have $G_2-G_1=L_1L_2L_3$, where, for $i \in
\{1,2,3\}$, 
the polynomial $L_i$ divides $H_1-\alpha_i H_2$, and hence $\Res(L_i,H_1)=\Res(L_i,\alpha_i H_2)$.  
We obtain that 
\[\Res(G_2-G_1,H_1)=\alpha_1\alpha_2\alpha_3\Res(L_1L_2L_3,H_2)
=\frac{\lambda_2}{\lambda_1}\Res(G_2-G_1,H_2),\]
which indeed implies that $e_3(T_1,T_2)=1$.
\end{proof}

\section{Level structure on principally polarized abelian surfaces}
\label{sec:level}

Let $n$ be a positive integer.  We need the analogues of the modular curves $Y(n)$, $Y_1(n)$ and $Y_0(n)$ 
for genus $2$ Jacobians. The theory is most conveniently stated in terms 
of slightly more general objects, namely \emph{principally polarized 
abelian surfaces} (PPAS). These include direct products of elliptic curves, 
equipped with the product polarization. The advantage is that the 
category of PPAS is closed under polarized isogenies.

Let $k$ be a field of characteristic not dividing $6n$. We 
write $\sA_2(1)$ for the moduli space of isomorphism classes of PPAS 
over $k$.  The coarse moduli space $\sA_2(1)$ is a $3$-dimensional 
variety.  Note however that, for any extension $L$ of $k$, the set $\sA_2(1)(L)$ 
of $L$-rational points of $\sA_2(1)$ corresponds to 
$L$-rational isomorphism classes that need not contain an $L$-rational 
abelian surface; similarly two abelian surfaces defined over $L$ that 
are isomorphic over $\kbar$ need not be isomorphic over $L$.

Let $\J$ be a principally polarized abelian surface over $k$. 
Then $\J[n](\kbar)$ has a non-degenerate bilinear alternating Galois 
covariant Weil-pairing $\J[n]\times\J[n]\to\mu_n$. A \emph{partial 
level $n$ structure} on $\J$ consists of a finite \'etale group 
scheme $\krn$ with a pairing $\krn\times\krn \to\mu_n$ and an injective 
homomorphism $\krn\to J[n]$ that is compatible with the pairings. An 
isomorphism between $(\J,\krn\to \J[n])$ and $(\J',\krn\to \J'[n])$ is 
an isomorphism $\phi:\J\to \J'$ of PPAS such that the composition 
of $\krn\to J[n]$ with $\phi$ yields $\krn\to J'$.
We write $\sA_2(\krn)$ for the moduli space of PPAS
equiped with a partial level $n$-structure involving $\krn$.  If we take for
example a sample abelian surface $\J_0$ and set $\krn=\J_0[n]$ 
then $\sA_2(\krn)$ is the moduli space of PPAS with full level $n$ structure.

We will work with the case $n=3$ and $\krn=\Z/3\times\Z/3$ with 
trivial pairing (i.e., $\krn$ is \emph{isotropic} with respect to its pairing).
Note that the automorphism group of $\krn$ is the full $\GL_2(\FF_3)$. 
However, since on every abelian surface $\J$ the involution 
$-1\colon \J\to\J$ induces the automorphism $-1$ on level structures, 
we deduce that $(\J,\krn\to \J)$ and $(\J,-\krn\to \J)$ are isomorphic as level $n$ 
structures on $\J$.  Therefore we obtain an action of $\PGL_2(\FF_3)$ on $\sA_2(\krn)$.

In our case $\krn\subset \J[3]$ is \emph{maximal isotropic}, so the nondegeneracy 
of $e_3$ yields that $J[3]/\krn\simeq \krn^\vee=\Hom(\krn,\mu_3)$. The 
fact that $\krn$ is maximal isotropic also means that the principal polarization 
on $\J$ induces a principal polarization on the isogenous abelian 
surface $\J/\krn$. Furthermore, the Weil pairing determines an 
injection $\krn^\vee\to (\J/\krn)[3]$. Thus we see that our maximal isotropic 
level structure leads to an isogeny $\J\to \J/\krn$, inducing an isomorphism 
$\sA_2(\krn)\to\sA_2(\krn^\vee)$, whose inverse is induced by the dual isogeny, 
using the principal polarizations to identify $\J$ and $\J/\krn$ with their duals.

The negation automorphism on $\J$ also gives rise to a \emph{quadratic 
twisting} operation. We write $\J^{(d)}$ for the twist of $\J$ by the 
quadratic character of discriminant $d$. A level structure under twisting 
gives rise to a twisted level structure $\krn^{(d)}\to \J^{(d)}$, 
where $\krn^{(d)}$ is the quadratic twist of $\krn$. This gives rise 
to an isomorphism $\sA_2(\krn)\to\sA_2(\krn^{(d)})$.

We finish by making some observations about the covering degrees of 
the various moduli spaces of level structures we have introduced. Let us 
write $\sA_2(3)$ for the space corresponding to full level $3$ data (say, 
for $\Z/3\times\Z/3\times\mu_3\times\mu_3$ with obvious pairing). This has 
a full $\PSp_4(\FF_3)$ acting on it. The subgroups fixing a maximal 
isotropic space are all conjugate and have the group structure $(\Z/3)^3$. 
There are $40$ of them. We see that $\sA_2(3)\to\sA_2(\krn)$ is finite of 
degree $27$. The cover $\sA_2(\krn)\to\sA_2(1)$ is of degree $40\cdot 24$, 
determined by the choice of isotropic space times the size of $\PGL_2(\FF_3)$.

The variety $\sA_2(3)$ is very well-known. Its completion is
the Burkhardt quartic, defined by the homogeneous equation
\begin{equation}\label{eq:burk}
t^4 - t(w^3+x^3+y^3+z^3) + 3wxyz = 0,
\end{equation}
as is described in~\cites{Burkhardt1891,hoffmanweintraub:siegel}. In particular, it, and its finite quotients, are
absolutely irreducible.

Note that most PPAS are Jacobians of genus $2$ curves in the sense that outside 
a proper closed subvariety of $\sA_2(1)$, any point can be represented by 
the Jacobian of a curve of genus $2$. In what follows we will determine 
a genus $2$ curve $\Crst$ over the function field of $\sA_2(\krn)$ together with a level structure on its Jacobian
that makes it correspond to the generic point on $\sA_2(\krn)$.

\section{Parametrisation of genus~2 curves with a maximal isotropic 
$(\ZZ/3)^2$ in $\J[3]$}\label{sec:param}

Let $k$ be a field of characteristic different from $2,3$.
In this section we derive a genus $2$ curve $\Crst$ over $k(r,s,t)$ with 
two non-trivial divisor classes $T_1,T_2\in \Pic(\Crst/k(r,s,t))[3]$ 
with $T_1\neq \pm T_2$ and $e_3(T_1,T_2)=1$. This specifies a $\krn$-level 
structure on the Jacobian $\Jrst$ of $\Crst$. In the process we will see 
that any sufficiently general Jacobian with a $\krn$-level structure 
over $k$ can be obtained via specialization from $\Jrst$. This identifies 
$k(r,s,t)$ with the function field of $\sA_2(\krn)$, verifying that this 
moduli space is indeed rational.

We use the notation from Lemmas~\ref{L:general_form} 
and \ref{L:trivial_pairing}. We consider the algebra 
$k[\alpha]=k[t]/(t^3-\lambda_2/\lambda_1)$, which is only a field 
if $\lambda_2/\lambda_1$ is not a cube in $k$, but at least will always be 
an \'etale algebra because $\lambda_1,\lambda_2\neq 0$. We write 
$\Nm=\norm_{k[x,\alpha]/k[x]}$. We have 
\[(G_2-G_1)(G_2+G_1)=\lambda_1 \Nm(H_1-\alpha H_2).\]
From Lemma~\ref{L:trivial_pairing} it follows that
\begin{equation}\label{E:H_factorization}
H_1-\alpha H_2=LM \text{ for some } L,M\in k[\alpha,x]
 \end{equation}
and that for some $c\in k^\times$ we have
\[\begin{aligned}
G_2-G_1&=\tfrac{1}{c} \Nm(M),\\
G_2+G_1&=c\lambda_1 \Nm(L),\\
\end{aligned} \quad \text{ and } \quad
\begin{aligned}
G_1&=\tfrac{1}{2}(c\lambda_1 \Nm(L)-\tfrac{1}{c} \Nm(M)),\\
G_2&=\tfrac{1}{2}(c\lambda_1 \Nm(L)+\tfrac{1}{c} \Nm(M)).
\end{aligned}
\]
We observe that
\[(c y)^2=(c G_1)^2+(c^2\lambda_1) H_1^3=(c G_2)^2+(c^2\lambda_2) H_2^3,\]
so by adjusting the values of $\lambda_1,\lambda_2$ we can assume $c=1$. 
This shows that the isomorphism class of $\sC$ is determined 
by $\lambda_1,\lambda_2,L,M$.

Furthermore, if $k$ has sufficiently many elements we can ensure that $L$ 
does not vanish at $x=\infty$ and that $L$ is monic, so that
\[L=x-(l_0+\alpha l_1+\alpha^2l_2).\]
The fractional linear transformation
\[x\mapsto \frac{l_1x+tl_2^2-l_0l_1}{l_2x+l_1^2-l_0l_2},\]
with determinant $l_1^3-tl_2^3$ sends $l_0+l_1\alpha+l_2\alpha^2$ 
to $\alpha$. One can check that if $l_1^3-tl_2^3=0$ then either $F$ has 
a repeated root, and hence our data does not specify a genus $2$ curve, 
or $l_1=l_2=0$. In the latter case, $L$ is already defined over $k$, so 
via $x\mapsto x-l_0$ we can move its root to $0$. This shows that via 
a fractional linear transformation, we can assume that
\begin{equation}\label{E:general_LM}
L=x-u\alpha\text{ and }M
=(c_0+c_1\alpha+c_2\alpha^2)x-(m_0+m_1\alpha+m_2\alpha^2),
\end{equation}
where $u=0$ corresponds to the case where $l_1=l_2=0$. 
In order for $LM$ to be of the form $H_1-\alpha H_2$, with $H_1,H_2\in k[x]$, 
we need
\begin{equation}\label{E:parameter_relations}
c_2=0,\ \ m_1u=0,\ \ m_2=-c_1u.
\end{equation}
We set $\lambda_1=s, \lambda_2=st$ and observe that $c^2F(x)$ is homogeneous 
with respect to the following gradings.
\[\begin{array}{r|rrrrrrrrrrr}
&s&t&c_0&c_1&c_2&m_0&m_1&m_2&u&x&c\\
\hline
\text{weights}
&3&0&-1&-1&-1&0&0&0&1&1&-3\\
&0&3&0&-1&-2&0&-1&-2&-1&0&0\\
&0&0&0&0&0&1&1&1&1&1&-3
\end{array}\]
We can solve \eqref{E:parameter_relations} via either $u=0$ or via $m_1=0$. 
For $u=0$ we find that $H_1$ and $H_2$ both have a root at $x=0$. By 
Lemma~\ref{L:disjoint} we know we can avoid this case by changing the 
basis for the $3$-torsion subgroup. Thus we see that any $\PGL_2(\FF_3)$-orbit 
has a representative that avoids this locus.

For the other case we take the affine open described by
\[(s,t,c_0,c_1,c_2,m_0,m_1,m_2,u)=(s,t,1,-1,0,-r,0,0,1),\]
leading to
\[H_1=x^2+rx+t,\;H_2=x^2+x+r,\;\lambda_1=s,\;\lambda_2=st.\]
It is instructive to record which cases are excluded by the choices 
that we make here. We use the gradings to scale $c_0=1,c_1=-1,u=1$, so any 
curves that require any of these parameters to be $0$ are ruled out. The 
gradings immediately show that setting any of $c_0,c_1,u$ to $0$ yields 
at most a $2$-dimensional family of curves. For $u=0$, we have seen 
that $H_1,H_2$ have a common root. Furthermore, the gradings show that 
this forms at most a $2$-dimensional family up to isomorphy. For $c_0=0$, 
we see that $H_1$ has a root at $\infty$ and for $c_1=0$ we see that $H_2$ 
has a root at infinity. Note that if $u\neq0$, we have applied a linear 
transformation to normalize the form of $L$, so $\infty$ has geometric 
meaning. For instance, Example~\ref{E:special_curve} describes 
a $1$-dimensional family of curves that lie in this locus.

To summarize, we have established the following theorem, where we have 
scaled $s$ by $4$ to avoid some denominators in coefficients.

\begin{thm}\label{T:rst_family}
Let $k$ be a field of characteristic different from $2,3$ and suppose that $(\C,T_1,T_2)$ 
consists of a genus $2$ curve $\C$ over $k$ and $T_1,T_2\in\Pic(\sC/k)[3]$ 
such that $\#\langle T_1,T_2\rangle=9$ and $e_3(T_1,T_2)=1$. If the specified 
data is sufficiently general then $(\C,T_1,T_2)$ is isomorphic to a suitable 
specialization of $r,s,t$ in the family described by the following data.
\[\begin{aligned}
H_1&=x^2 + rx + t\\
\lambda_1&=4s\\
G_1&=(s-st - 1)x^3 + 3s(r-t)x^2 + 3sr(r-t)x - st^2 + sr^3 + t\\[0.5em]
H_2&=x^2+x+r\\
\lambda_2&=4st\\
G_2&=(s-st + 1)x^3 + 3s(r-t)x^2 + 3sr(r-t)x - st^2 + sr^3 - t\\[0.5em]
H_3&=sx^2 + (2sr -st - 1)x + sr^2\\
\lambda_3&=4t/(st+1)^2\\
G_3&=\big((s^2t^2 - s^2t + 2st + s + 1)x^3 
+ (3s^2t^2 - 3s^2tr + 3st + 3sr)x^2 \\
&+ (3s^2t^2r - 3s^2tr^2 + 3str + 3sr^2)x 
+ s^2t^3 - s^2tr^3 + 2st^2 + sr^3 + t\big)/(st+1)\\[0.5em]
H_4&=(str - st - sr^2 + sr + r)x^2 + (st^2 - str - st - sr^3 + 2sr^2 + t)x 
+ st^2 - str^2 - str + sr^3 + t\\
\lambda_4&=4st/\big(
st^2 - 3str + st + sr^3 + t
\big)^2\\
G_4&=\big(
(s^2t^3 - 3s^2t^2r - s^2tr^3 + 6s^2tr^2 - 3s^2tr + s^2t - s^2r^3 
+ 2st^2 - 3str + 2st - sr^3 + t)x^3\\
&+ (3s^2t^3 - 6s^2t^2r^2 - 3s^2t^2 + 9s^2tr^3 - 3s^2tr^2 + 3s^2tr 
- 3s^2r^4 + 3st^2 - 6str^2 + 3str)x^2\\
&+ (-3s^2t^3r + 6s^2t^3 - 9s^2t^2r + 3s^2tr^4 + 3s^2tr^3 + 3s^2tr^2 
- 3s^2r^5 - 3st^2r + 6st^2 - 3str^2)x\\
&- s^2t^4 + 3s^2t^3r + s^2t^3 - 6s^2t^2r^2 + 3s^2tr^4 + s^2tr^3 
- s^2r^6 - 2st^3 + 3st^2r + st^2 - 2str^3 - t^2
\big)/\\
&(st^2 - 3str + st + sr^3 + t).
\end{aligned}
\]
Here $\Crst\colon y^2=F_{rst}(x)=G_i^2+\lambda_i H_i^3$ for $i=1,2,3,4$ 
and $T_i=[\{H_i(x)=0,y-G_i(x)=0\}-\kappa]$ and $T_3=T_1+T_2$ and $T_4=T_1-T_2$.
\end{thm}

For future reference we note that
\begin{equation}\label{eq:discF}
\Disc(F_{rst})=
-2^{12}3^6\delta_1^3\delta_2^3\delta_3\delta_4^3\delta_5\delta_6^3\delta_7^3,
\end{equation}
where
\begin{equation}\label{eq:discFfactors}
\begin{aligned}
    \delta_1 &= s\\
    \delta_2 &= t\\
    \delta_3 &= st + 1\\
    \delta_4 &= r^3 - 3rt + t^2 + t\\
    \delta_5 &= r^3s - 3rst + st^2 + st + t\\
    \delta_6 &= r^3s^2 - 3rs^2t - 3rs + s^2t^2 + s^2t + 2st + s + 1\\
    \delta_7 &= r^3s^2t + r^3s - 3rs^2t^2 - 3rst + s^2t^3 + s^2t^2 + 2st^2 + t. 
\end{aligned}
\end{equation}

\begin{rmk}\label{R:Crst_twists}
Note that $\Crst^{(d)}\colon dy^2=F_{rst}(x)$ has exactly the same property 
for $\krn^{(d)}$-level structure.
\end{rmk}

\section{The isogeny}\label{sec:isog}

We consider the curve $\Crst$ as defined in Theorem~\ref{T:rst_family} and its Jacobian $\Jrst$. In this section we determine a curve $\tCrst$ whose Jacobian $\tJrst$ is isogenous to $\Jrst$ via the isogeny $\Jrst\to \Jrst/\krn$.
We do so by determining the corresponding map between their Kummer surfaces.

\subsection{Isogenies and the quartic model of the Kummer surface}

Let $\J$ be a principally polarized abelian surface with theta divisor 
$\Theta_\J$ and suppose that $\krn\subset \J[3]$ is maximal isotropic. 
We consider $\tJ=\J/\krn$ and the isogeny $\phi\colon \J\to\tJ$. By 
\cite{milne:abvar}*{Proposition~16.8} there exists a principal polarization 
on $\tJ$ with theta divisor $\Theta_\tJ$ such that 
$\phi^*(\Theta_\tJ)=3\Theta_\J$.

The classical theory of theta divisors gives us that $\sO_\J(2\Theta_\J)$ 
is $4$-dimensional and that the induced map $\sJ\to\PP^3$ yields a 
quartic model of the Kummer surface $\K=\J/\langle -1\rangle$.
We write $\tK$ for the Kummer surface of $\tJ$. Similarly, the linear system 
$\sO_\tJ(2\Theta_\tJ)$ provides a quartic model of $\tK$.

\begin{rmk}
Note that this construction requires that we choose $2\Theta_\J$ to be 
defined over the base field $k$. If $k$ is not algebraically closed, 
then there might not exist a $k$-rational divisor $\Theta_\J$ in its 
class. For an abelian surface, however, $2\Theta_\J$ is always linearly 
equivalent to a $k$-rational divisor.
\end{rmk}

The isogeny induces a map $\K \to \tK$ in the following way. First 
note that $\phi^* \sO_\tJ(2\Theta_{\tJ})\subset \sO_\J(6\Theta_{\J})$. 
The involution $-1\colon \J\to\J$ induces a linear map on 
$\sO_\J(6\Theta_{\J})$. We write $\sO_\J(6\Theta_{\J})^+$ for the fixed 
subspace. Similarly, the translation action of $\krn$ on $\J$ induces 
a linear action on the same space. We write $\sO_\J(6\Theta_{\J})^\krn$ for 
its fixed space. It is straightforward to check that
\[
\phi^* \sO_\tJ(2\Theta_{\tJ}) 
= \sO_\J(6\Theta_\J)^\Sigma \cap \sO_\J(6\Theta_\J)^+.
\]
If $\xi_0,\xi_1,\xi_2,\xi_3$ forms a basis for $\sO_\J(2\Theta_\J)$ 
then $\sO_\J(6\Theta_\J)^+$ is generated by the cubic forms in 
$\xi_0,\xi_1,\xi_2,\xi_3$. Thus we see that the isogeny $\phi\colon \J\to\tJ$ 
induces map $\K\to\tK$ which, between the quartic models, is given 
by cubic forms.

\subsection{Choice of model for Kummer surfaces}

Let $\C$ be a curve of genus $2$ given by a model
\[
 \C\colon y^2=f_6x^6+f_5x^5+\cdots+f_0,
\]
with Jacobian $\J$. We follow \cite{prolegom}*{p.17} and choose 
a particular basis for $\sO_\J(\Theta_\J)$. 
We describe $\xi$ $=$ $\xi(D)$ $=$ $\xi_0,\ldots,\xi_3$ 
as functions on $\J$ in terms of a divisor class
$D = [(x_1,y_1)+(x_2,y_2) - \kappa]$ on $\C$ as follows.
\begin{equation}\label{eq:kummercoords}
\xi_0 = 1,\quad \xi_1 = x_1 + x_2,\quad  \xi_2 = x_1 x_2,\quad 
\xi_3 = \frac{\Phi(\xi_0,\xi_1,\xi_2) - 2y_1y_2 }{\xi_1^2-4\xi_0\xi_2}, 
\text{ where}
\end{equation}
\begin{equation*}
\Phi(\xi_0,\xi_1,\xi_2) = 
2f_0\xi_0^3 + f_1\xi_0^2\xi_1 + 2f_2\xi_0^2\xi_2
+ f_3\xi_0\xi_1\xi_2+ 2f_4\xi_0\xi_2^2 + 
f_5\xi_2^2\xi_1 + 2f_6\xi_2^3.
\end{equation*}
Note that for a Mumford representation 
$D=[\{x^2-\xi_1x+\xi_2, y-g_0-g_1x\}-\kappa]$ we
have $y_1y_2=g_0^2+g_0g_1\xi_1+g_1^2\xi_2$, so one can compute these 
coordinates readily from such a representation.

The quartic equation for the model of $\K$ arising from these 
coordinates has the shape
\[
\K\colon(\xi_1^2-4\xi_0\xi_2)\xi_3^2
+\Phi(\xi_0,\xi_1,\xi_2)\xi_3+\Psi(\xi_0,\xi_1,\xi_2)=0,
\]
where $\Psi(\xi_0,\xi_1,\xi_2)$ is a quartic form we do not need 
explicitly here. The important observation is that one can read 
off the coefficients $f_0,\ldots,f_6$ directly from $\Phi$ and 
thus recover $\C$ from it. 

In order to produce $\krn$-invariant forms on $\K$, we use biquadratic 
forms from \cite{prolegom}*{p.23}, arising from the addition structure 
on $\J$. For $i,j = 0,\ldots ,3$ we have forms 
$B_{i,j}\in k[\xi_0,\ldots,\xi_3,\xi_0',\ldots,\xi_3']$, biquadratic 
in $(\xi_0,\ldots,\xi_3)$ and $(\xi_0',\ldots,\xi_3')$ such that for 
points $D_1,D_2$ on $\J$ we have, as projective matrices,
\begin{equation}\label{eq:biquads}
\Bigl( \xi_i(D_1 + D_2)\, \xi_j(D_1 - D_2)
+ \xi_i(D_1 - D_2)\, \xi_j(D_1 + D_2) \Bigr)
= \Bigl( B_{ij}(\xi(D_1), \xi(D_2)) \Bigr). 
\end{equation}
We fix two points $T_1,T_2\in \J[3]$ that generate $\krn$, 
write $\xi(T_1),\xi(T_2)$ for the coordinate vectors of their 
images on the quartic model of $\K$ and define
\begin{equation}\label{eq:RS}
R_{ij}\bigl(\xi_0,\ldots,\xi_3 \bigr) 
= B_{ij}\bigl(\xi_0,\ldots,\xi_3, \xi(T_1) \bigr)\text{ and }
S_{ij}\bigl(\xi_0,\ldots,\xi_3 \bigr) 
= B_{ij}\bigl(\xi_0,\ldots,\xi_3, \xi(T_2) \bigr).
\end{equation}
We see that the cubic forms
\[R_{ijk}=\xi_iR_{jk}+\xi_j R_{ki}+ \xi_k R_{ij} 
\text{ with }i,j,k \in\{ 1,\ldots ,4\}\]
are invariant under translation by $T_1$ and similarly that the forms
\[S_{ijk}=\xi_iS_{jk}+\xi_j S_{ki}+ \xi_k S_{ij} 
\text{ with }i,j,k \in\{ 1,\ldots ,4\}\]
are invariant under translation by $T_2$. For $\C=\Crst$ the $R_{ijk}$ 
and $S_{ijk}$ each generate spaces of dimension $8$ that intersect 
in a space of dimension $4$. This intersection provides us with an 
explicit description of $\phi^*(\sO_\tJ(2\Theta_\tJ))$.

Generally, we expect $\tJ$ to be the Jacobian of a curve of genus $2$, 
say $\tC$.  We can try to find a basis $\txi_0,\ldots,\txi_3$ 
for $\phi^*(\sO_\tJ(2\Theta_\tJ))$ that is the pullback of a basis 
of the type described by \eqref{eq:kummercoords}. We can then read off 
the curve $\tC$, at least up to quadratic twist, from the resulting 
equation for $\tK$. The basis choice can largely be characterized 
by the order of vanishing of each $\xi_i$ at the identity element. 
This leads us to conclude that, up to scalar multiples, we should 
take the basis choice
\[
\begin{aligned}
\txi_0&=(1\xi_0+0\xi_1+0\xi_2)\xi_3^2+\cdots,\\
\txi_1&=(0\xi_0+1\xi_1+0\xi_2)\xi_3^2+\cdots,\\
\txi_2&=(0\xi_0+0\xi_1+1\xi_2)\xi_3^2+\cdots.\\
\end{aligned}
\]
The determination of $\txi_3$ is a little more involved. The resulting 
forms for $\tC=\tCrst$ are too voluminous to reproduce here, but we 
have made them available electronically at \cite{bft:electronic}. Via 
interpolation we find tentatively the following result.
\begin{thm}\label{T:Ctilde}
  Let $\Crst$ be as described by Theorem~\ref{T:rst_family}. 
Then $\tJrst=\Jrst/\krn$ is the Jacobian of the genus $2$ curve
\[\tCrst\colon -3y^2= \tG_4^2+\tlambda_4 \tH_4^3,\]
with
\[\begin{aligned}
\tG_4&=\Delta\,\big((s-st-1)x^3+3s(r-t)x^2+3rs(r-t)x+(r^3s - st^2 - t)\big),\\
\tH_4&=(r-1)(rs-st-1)x^2+(r^3s-2r^2s+rst+r-st^2+st-t)x-(r^2-t)(rs-st-1),\\
\tlambda_4&=4\Delta st,
\end{aligned}\]
where
\begin{equation}\label{eq:Delta}
\begin{split}
\Delta &= r^6 s^2 - 6 r^4 s^2 t - 3 r^4 s + 2 r^3 s^2 t^2 + 2 r^3 s^2 t +
3 r^3 s t
 + r^3 s + r^3 + 9 r^2 s^2 t^2 + 6 r^2 s t\\
&\quad - 6 r s^2 t^3 - 6 r s^2 t^2 - 9 r s
t
^2 - 3 r s t - 3 r t + s^2 t^4 + 2 s^2 t^3 + s^2 t^2 + 2 s t^3 + 3 s t^2 + t^2 +
 t
\end{split}
\end{equation}
\end{thm}

\subsection{Proof of Theorem~\ref{T:Ctilde}}

Since we have completely explicit descriptions of $\Crst$ and $\tCrst$, we can write down explicit quartic models
\[\begin{aligned}
\K\colon&Q=(\xi_1^2-4\xi_0\xi_2)\xi_3^2+\Phi(\xi_0,\xi_1,\xi_2)\xi_3+\Psi(\xi_0,\xi_1,\xi_2)=0\\
\tK\colon&\tilde{Q}=(\txi_1^2-4\txi_0\txi_2)\txi_3^2+\tPhi(\txi_0,\txi_1,\txi_2)\txi_3+\tPsi(\txi_0,\txi_1,\txi_2)=0.
\end{aligned}\]
Furthermore, we have an explicit description at \cite{bft:electronic} of the 
map between them by expressions that give $\txi_0,\ldots,\txi_3$ as cubic forms 
in $\xi_0,\ldots,\xi_3$. We already know that $\tK$ is irreducible, 
because it is the Kummer surface of a Jacobian. Therefore, to check 
if $\tK$ is indeed the image of $\K$, we only need to substitute the 
cubic forms into the equation for $\tK$ and check that the resulting 
degree $12$ equation is divisible by the quartic equation for $\K$. 
This is doable for specific specializations of $r,s,t$ 
in $\QQ$, but the computers at our disposal were not able to do this directly.

We note that $Q$ and $\tilde{Q}$ (after substitution of the cubic forms), 
are polynomials in $r,s,t,\xi_0,\ldots,\xi_3$, of degrees $10$ and $2$ 
in $\xi_3$ respectively. Hence, long division yields unique polynomials 
$\sigma\in \QQ[r,s,t,\xi_0,\ldots,\xi_3]$ 
and $\rho_0,\rho_1\in \QQ[r,s,t,\xi_0,\xi_1,\xi_2]$ such that
\[(\xi_1^2-4\xi_0\xi_2)^9 \tilde{Q} = \sigma Q +\rho_1 \xi_3 + \rho_0.\]
We want to prove that $\rho_1$ and $\rho_0$ are identically zero. To 
this end, we analyse the appropriate Newton polygons (or do the required 
computation using polynomials with coefficients truncated to the 
appropriate leading terms) to verify that $\rho_0,\rho_1$ are of degrees 
at most $102, 67, 36$ in $r,s,t$. Hence, if we check that $Q$ indeed 
divides $\tilde{Q}$ for a grid of $103\times 68\times 37$ values 
for $(r,s,t)$ then a straightforward interpolation argument shows that 
$\rho_0,\rho_1$ must indeed be identically $0$. This is something that 
can easily be verified by a computer in less than $3$ hours.

This computation shows that $\tK$ is indeed the Kummer surface of 
$\tJ=\J/\krn$ and hence that $\tCrst$ is correct up to quadratic 
twist. Recall from Section~\ref{sec:level}
that $\J/\krn$ comes equipped with a $\krn^\vee$-level structure. In our 
case, we have that $\krn=(\ZZ/3)^2$, so $\krn^\vee=(\mu_3)^2=\krn^{(-3)}$. 
Thus, it follows that $\tJrst^{(-3)}$ should have a $\krn$-level structure 
itself such that the isogeny corresponding to it brings us back 
to $\Jrst^{(-3)}$.

\begin{lemma}\label{L:psi0}
Let $\sC_{rst}$ be as in Theorem~\ref{T:rst_family},
let $\tCrst$ be as in Theorem~\ref{T:Ctilde}, and let
$\tCrst^{(-3)}$ be the quadratic twist of~$\tCrst$ by~$-3$,
using the notation in Remark~\ref{R:Crst_twists}.
Define $\psi_0$ by
\[ 
\psi_0(r,s,t) = 
\Bigl( \frac{-s(r-1)(r^2-t)(\delta_5-r)}{(rs-st-1)^2\delta_4},\
         \frac{(rs-st-1)^3 \delta_4^2}{st(r-1)^3\Delta},\
         \frac{s^2 (r-1)^3 (r^2-t)^3}{(rs-st-1)^3\delta_4^2} \Bigr).
\] 
Then $\sC_{r's't'}$ is birationally equivalent
to $\tCrst^{(-3)}$, where $(r',s',t') = \psi_0(r,s,t)$. Furthermore, as a rational map we have $\psi_0(\psi_0(r,s,t))=(r,s,t)$.
The $\Sigma^{(-3)}$ level structure induced on $\tJrst$ determines the kernel of the dual isogeny $\tJrst\to\Jrst$
\end{lemma}
\begin{proof}
One can check directly that $\sC_{r's't'}$ is birationally
equivalent to $\tCrst^{(-3)}$ under the transformation
\[ 
\theta_0 : (x,y) \mapsto 
\Bigl( \frac{-(r^3 - 3rt + t^2 + t)(rs - st - 1)}{(r^2-t)(r-1)^2 s} x
+ \frac{r-t}{r-1},\
  \frac{\Delta t (rs - st - 1)^3 (r^3 - 3rt + t^2 + t)^2 }
       {s^2 (r-1)^3 (r^2-t)^3} y\Bigr).
\]
This naturally marks some $\krn^{(-3)}$ level structure on $\tJrst$. Note that it even does so over $\QQ$, where we have no primitive cube root of unity. The Weil pairing implies that on $\tJrst^{(-3)}$, any two $\krn$ level structures must differ by a unique automorphism of $\krn$. It follows that the same holds for $\krn^{(-3)}$ level structures on $\tJrst$ itself.
\end{proof}

\begin{rmk}
A little more is true than we prove in Lemma~\ref{L:psi0}: we have a \emph{natural} $\krn^\vee$ structure on $\tJrst$. In Lemma~\ref{L:sigma_vee_id}, we identify this and in Lemma~\ref{L:psi123} we identify the corresponding involution on $k(r,s,t)$, which is not \emph{quite} $\psi_0$ as listed above. We selected $\psi_0$ because the corresponding transformation $\theta_0$ is easy to write down.
\end{rmk}

\subsection{Additional relations}

At this point, we have what we require for the applications
of the next sections, since we only need the $\Sigma$ level structure up to $\Sigma$-automorphism.
We shall devote the remainder of the section to a more concise description of~$\tCrst$,
which will also give the full natural $\krn^{\vee}$ level structure on $\tJrst$.

\begin{rmk} In the case of isotropic $(\ZZ/2)^2$-level structure, there 
is a very satisfying expression for the isogenous abelian variety in the 
general case (i.e., when it is a Jacobian), described 
in~\cite{bostmestre:richelot} (see also \cite{prolegom}*{Chapter~9}). 
An isotropic $(\ZZ/2)^2$-level structure on the Jacobian of a genus~$2$ 
curve can be expressed by a model of the curve of the form
\[\C\colon y^2 = q_1(x) q_2(x) q_3(x),\]
where each $q_i$ is a quadratic polynomial in $x$. One forms a $3\times 3$ 
matrix whose columns are the coefficients of $q_1,q_2,q_3$. If the determinant $\Delta$ 
of this matrix is nonzero, then the isogenous surface is a Jacobian 
and the associated curve can be expressed as
\[\tC\colon y^2=\Delta \tq_1(x)\tq_2(x)\tq_3(x),\]
where the coefficients of the $\tq_i$ are easily expressible in terms 
of the cofactors of this same $3\times 3$ matrix. Of particular note 
is that the curve is (naturally) again of the same form, Indeed, it is 
straightforward to verify that the same operation applied twice gives us 
back a model that is isomorphic to the curve we started with and that 
the quadrics satisfy the peculiar relation
\[q_1(x) \tq_1(\tx) + q_2(x) \tq_2(\tx) +  q_3(x) \tq_3(\tx)=
\Delta (x - \tx)^2.\]
In fact, this relation is the basis for the $(2 , 2)$ correspondence 
between $\C$ and $\tC$ that gives rise to the polarized isogeny between 
their Jacobians.
\end{rmk}

One might hope to find a similar relation in our case. Indeed, the general 
theory 
implies there is a correspondence 
between $\Crst$ and $\tCrst$ giving rise to the polarized isogeny 
between $\Jrst$ and $\tJrst$. However, that general theory only predicts 
an $(18 , 2)$ correspondence which lacks the desired symmetry and 
does not seem inviting from a computational point of view. The following 
theorem gives a possibly more attractive relation between the models 
for $\Crst$ and $\tCrst$ expressing the level structures on their Jacobians.

\begin{thm}\label{thm:Amtx}
For $j=1,\ldots, 4$ let $H_j(x)=h_{2j}x^2+h_{1j}x+h_{0j}$ be 
as in Theorem~\ref{T:rst_family};
for $i=1,\ldots ,7$, let $\delta_i$, be as in~(\ref{eq:discFfactors}).  
Define the matrix~$A$ by 
\begin{equation}\label{eq:Amtx}
A = \begin{pmatrix}
              h_{21} & h_{22} & h_{23} & h_{24} \\
              h_{11} & h_{12} & h_{13} & h_{14} \\
              h_{01} & h_{02} & h_{03} & h_{04} \\
       -\frac{r}{2} & -\frac{1}{2} &
       -\frac{\delta_3}{2} & -\frac{\delta_5}{2}
     \end{pmatrix}.
\end{equation}
Then $\det(A)=\Delta$ as in \eqref{eq:Delta}. Let $M$ be the cofactor 
matrix of $A$, i.e., $M^T= \det(A) A^{-1}$ and let $M_{ij}$ be its 
entry in the $i$-th row and $j$-th column.
Define
\begin{equation}\label{eq:Atildemtx}
\tA = \left(\!\!\!\!\begin{array}{rrrr}
   M_{31} & M_{32} & M_{33} & M_{34} \\
   -2 M_{21} & -2 M_{22} & -2 M_{23} & - 2 M_{24} \\
   M_{11} & M_{12} & M_{13} & M_{14} \\
  \frac{1}{2}M_{41} &\frac{1}{2}M_{42} &\frac{1}{2}M_{43} &\frac{1}{2}M_{44} 
          \end{array}\!\!\right),
\end{equation}
and
\begin{equation}\label{eq:HtildefromAtilde}
\tH_j(x) = \tA_{1j} x^2 + \tA_{2j} x + \tA_{3j},\hbox{ for } j = 1,\ldots 4,
\end{equation}
so that the $\tH_j(x)$ bear the same relationship to the first
three rows of $\tA$ as the $H_j(x)$ bear to~$A$.
Also define
\begin{equation}\label{eq:lambdahats}
\tlambda_1 = \lambda_1 \Delta / \delta_6^2, \ \
\tlambda_2 = \lambda_2 \Delta / \delta_7^2, \ \
\tlambda_3 = \lambda_3 \Delta \delta_3^2 / \delta_4^2, \ \
\tlambda_4 = \lambda_4 \Delta \delta_5^2.
\end{equation}
Finally define
\begin{equation}\label{eq:G1hat}
\tG_4(x) = \Delta \bigl( G_1(x) - 2 t \bigr),
\end{equation}
and define $\tG_1(x),\tG_2(x),\tG_3(x)$ up to $\pm$
(which is all we require for these) to be such that
\begin{equation}\label{eq:Gtildes}
\tG_i(x)^2 = 
\tG_4(x)^2 + \tlambda_4 \tH_4(x)^3 - \tlambda_i \tH_i(x)^3,
\hbox{ for } i = 1,2,3.
\end{equation}
Then the curve $\tCrst$ of Theorem~\ref{T:Ctilde} is the same as
\begin{equation}\label{Ctildeagain}
\tCrst : y^2 = 
-3 \Bigl( \tG_i(x)^2 + \tlambda_i \tH_i(x)^3 \Bigr),
\hbox{ for } i = 1,\ldots , 4.
\end{equation}
\end{thm}
An immediate consequence the relationship between the
matrices $A, \tA$ is the following identity, which
is strikingly similar to
the identity~(9.2.5) of~\cite{prolegom} for the Richelot isogeny.
\begin{cor}\label{cor:magic}
Let the $H_i(x), \tH_i(x)$ be as in Theorem~\ref{thm:Amtx}. Then
\begin{equation}\label{eq:magic}
H_1(x) \tH_1(\tx) + H_2(x) \tH_2(\tx) 
+ H_3(x) \tH_3(\tx) + H_4(x) \tH_4(\tx)
 = \Delta (x - \tx)^2.
\end{equation}
\end{cor}
\begin{proof}
We first note that, if we take the matrix~$\tA$ of~(\ref{eq:Atildemtx}),
divide the second row by~$-2$, then swap the first and third rows,
and then take the transpose, we obtain $\adj{A}$,
the adjugate of the matrix~$A$ of~(\ref{eq:Amtx}).
We recall the standard identity $A \adj{A} = \Delta I_4$
from linear algebra and note that, in the left hand side
of~(\ref{eq:magic}), the coefficients of $x^2, \tx^2$
are equal to diagonal entries of $A \adj{A}$ and so equal~$\Delta$.
Similarly, the coefficient of $x\tx$
is $-2$ times a diagonal entries of $A \adj{A}$ and 
so equals~$-2\Delta$. The remaining coefficients on the
left hand side of~(\ref{eq:magic}) equal non-diagonal entries
of $A \adj{A}$, and so are all~$0$, as required.
\end{proof}

We can also formulate the relationship between the $H_i$ and the $\tH_i$ in
more intrinsic terms. Let $T_1,T_2,T_3,T_4\in \Jrst[3]$ corresponding to $H_i$ such that
$T_1+T_2=T_3$ and $T_1-T_2=T_4$ and let $\tT_i$ and $\tH_i$ be related analogously. Suppose we have a basis $T_1,T_2,U_1,U_2$ for $\Jrst[3]$ such that the Gram matrix of the Weil pairing (with values written additively) is
\[\begin{pmatrix}0&0&0&1\\0&0&-1&0\\0&1&0&0\\-1&0&0&0\end{pmatrix}.\]
Then $U_1,U_2$ naturally provides a basis choice for $\Jrst[3]/\krn=\krn^\vee$. We find that $\phi(U_1)=\pm \tT_1$ and that $\phi(U_2)=\pm \tT_2$, with the same choice of sign (which is an ambiguity in $U_1,U_2$ already). This provides the following.

\begin{lemma}\label{L:sigma_vee_id}
The labelling for the $\tT_i$ naturally marks the $\Sigma^\vee$ level structure on $\tJrst$ in the sense that for $i=1,\ldots,4$, we have
\[e_3(T_i,U)=0\text{ for any }U\in\tJrst[3] \text{ such that }\phi(U)=\tT_i\]
\end{lemma}
\begin{proof}
We only have to check this statement for a particular specialization of $r,s,t$, for instance over a finite field where the full $3$-torsion is pointwise defined. We have a completely explicit description of the map $\phi$ on the Kummer surfaces, which is sufficient to determine the appropriate elements $U$. One can then just verify the claim in the lemma by exhaustion. It is straightforward to check that the condition given indeed uniquely determines the structure (up to sign).
\end{proof}

\subsection{Automorphisms}

Recall from Section~\ref{sec:level} that $\sA_2(\krn)$ has $\PGL_2(\FF_3)$ acting on it. Furthermore, because in our case we have $\krn^\vee\simeq \krn^{(-3)}$, we get an additional automorphism $\J\mapsto \tJ^{(-3)}$. We identified the effect of this last automorphism on $r,s,t$ in Lemma~\ref{L:psi0}. Here we describe generators for the other automorphisms as well. Note that $\psi_0$ is only a rational map, because the abelian surface $\tJ$ need not be a Jacobian if $\J$ is: $\tJ$ may be a product of elliptic curves. In addition, while we have seen in Section~\ref{sec:param} that every Jacobian with a $\krn$ level structure admits a model that is a specialization of $\Jrst$, this may involve a change of basis. Thus, we should also expect $\PGL_2(\FF_3)$ to only act birationally on $(r,s,t)$.

\begin{lemma}\label{L:psi123}
The following transformations
\[ \begin{aligned}
\psi_1(r,s,t) &= \Bigl( \frac{t}{r^2},\ \frac{r^3 s}{t},\ \frac{t^2}{r^3} 
\Bigr),\\
\psi_2(r,s,t) &= \Bigl( r,\ \frac{1}{s(r^3 - 3rt + t^2 + t)},\ t \Bigr),\\
\psi_3(r,s,t) &= \Bigl( r,\ \frac{t(st+1)}{r^3},\ \frac{r^3 s}{st + 1} \Bigr),
\end{aligned}
\]
have the property that, for each $i = 1,2,3$, 
if $(r',s',t') = \psi_i(r,s,t)$ then $\sC_{r's't'}$ 
is birationally equivalent to $\sC_{rst}$. The group generated
by $\psi_1,\psi_2,\psi_3$ is isomorphic to $\PGL_2(\FF_3)$.
Furthermore, if $\psi_0$ is as in Lemma~\ref{L:psi0}
then the group generated by $\psi_0,\psi_1,\psi_2,\psi_3$ 
is isomorphic to $\Z/2 \times \PGL_2(\FF_3)$.
\end{lemma}
\begin{proof}
We first note that, for each $i=1,2,3$, 
if $(r',s',t') = \psi_i(r,s,t)$ then $\sC_{r's't'}$
is birationally equivalent to $\sC_{rst}$ under $\theta_i$,
where 
\[ \begin{aligned}
\theta_1(x,y) &= \Bigl( \frac{t}{rx},\ \frac{t y}{x^3} \Bigr),\\
\theta_2(x,y) &= 
\Bigl( \frac{(r-t) x + (r^2-t)}{(r-1) x + (t-r)} ,\ 
\frac{s(r^3 - 3rt + t^2 + t)^2 y}{(rx - x + t - r)^3} \Bigr),\\
\theta_3(x,y) &=
\Bigl( \frac{-r x}{x+r} ,\ \frac{r^3 y}{(x+r)^3} \Bigr).
\end{aligned}
\]
It can also be checked that $\psi_1,\psi_2,\psi_3$ permute
the roles of the $H_i$ and correspond, respectively, to:
\[ \begin{aligned}
( H_1, H_2, H_3, H_4) &\leftrightarrow ( H_2, H_1, H_3, H_4),\\
( H_1, H_2, H_3, H_4) &\leftrightarrow ( H_1, H_2, H_4, H_3),\\
( H_1, H_2, H_3, H_4) &\leftrightarrow ( H_3, H_2, H_1, H_4).
\end{aligned}
\]
It follows that they can be identified with the transpositions
$(12), (34), (13)$ in $S_4$, which generate all of~$S_4$.
Hence the group generated by $\psi_1,\psi_2,\psi_3$ is isomorphic
to~$S_4$ which, in turn, is isomorphic to $\PGL_2(\FF_3)$.
Note also that $\psi_1,\psi_2,\psi_3$ give the same permutation
of the roles of the $\tH_i$. 
We finally note that $\psi_0$ corresponds to
$( H_1, H_2, H_3, H_4) \leftrightarrow (\tH_4, \tH_3, \tH_2, \tH_1 )$  
and can be replaced by $\psi_0' = \psi_3  \psi_1  \psi_2 \psi_3 \psi_0$
which corresponds to
$( H_1, H_2, H_3, H_4) \leftrightarrow (\tH_1, \tH_2, \tH_3, \tH_4 )$,
and is given explicitly by
\[ 
\psi'_0(r,s,t) = 
\Bigl( \frac{-(r^2-t)(rs-st-1)(\delta_5-r)}{(r-1)^2\delta_7},\
       \frac{(r-1)^3 s \delta_6 \delta_7^2}{(rs-st-1)^3 \delta_5 \Delta},\
       \frac{t(rs-st-1)^3(\delta_5-r)^3}{(r-1)^3\delta_6\delta_7^2} \Bigr).
\]
This is an involution and it commutes with $\psi_1,\psi_2,\psi_3$.
Hence the group generated by the maps $\psi_0,\psi_1,\psi_2,\psi_3$,
which is the same as the group generated by $\psi'_0,\psi_1,\psi_2,\psi_3$,
must be isomorphic to the group $\Z/2 \times \PGL_2(\FF_3)$.
\end{proof}

\section{Isogeny descent}
\label{sec:descent}

Galois cohomology associates to an isogeny
\[0\to \J[\phi]\to \J \stackrel{\phi}{\to} \tJ\to 0\]
between abelian varieties over a field $k$ an exact sequence
\[0\to\frac{\tJ(k)}{\phi \J(k)}\stackrel{\gamma}{\to} 
H^1(k,\J[\phi])\to H^1(k,\J).\]
For $k$ a number field and $v$ a place of $k$, we consider the 
completion $k_v$ and its separable closure $k_v^\sep$ and identify 
$\Gal(k_v^\sep/k_v)$ with a relevant decomposition group inside 
$\Gal(k^\sep/k)$. This allows us to consider restriction maps 
$\res_v\colon H^i(k,.)\to H^i(k_v,.)$.
Writing $\gamma_v$ for the relevant connecting homomorphism over 
the base field $k_v$, this allows us to define the \emph{Selmer group}
\begin{equation}\label{E:selmerdef}
\Sel^{\phi}(\J/k)=\{\delta \in H^1(k,\J[\phi]): 
\res_v(\delta)\in \mathrm{im} \gamma_v\text{ for all places $v$ of $k$}\}.
\end{equation}
The Selmer group contains the image of $\gamma$. If this containment 
is strict then part of the Selmer group represents non-trivial elements 
in $\Sha(\J/k)$. To be precise, we have
\[0\to \frac{\tJ(k)}{\phi \J(k)}\to 
\Sel^{\phi}(\J/k)\to \Sha(\J/k)[\phi]\to 0.\]
Therefore, the computation of Selmer groups can be used to exhibit 
non-trivial elements in Tate-Shafarevich groups. This is taking a 
historically backward view, since originally Tate-Shafarevich groups 
were introduced as a means to measure the failure of Selmer groups 
to provide sharp bounds on the size of ${\tJ(k)}/{\phi \J(k)}$.

Let $\GG_m$ be the multiplicative group scheme over a field $k$ of 
characteristic not dividing $2$. For $d\in k^\times$, we write
$\GG_m^{(d)}$ for the \emph{quadratic twist by $d$} of the multiplicative 
group. It is the group scheme that fits in the short exact sequence
\[1\to \GG_m^{(d)}(L) \to L[\sqrt{d}]^\times 
\xrightarrow{\operatorname{Norm}} L^\times\to 1\]
for any extension $L$ of $k$. Similarly, for a positive integer $n$, we write 
$\mu_n^{(d)}\subset \GG_m^{(d)}$ for the kernel of the morphism $x\to x^n$.

We begin by stating the following slight generalization of a classical 
result from Kummer theory.
\begin{lemma} Let $n>0$ be odd, let $k$ be a field of characteristic 
not dividing $2n$ and let $\mu_n$ be the $\Gal(k^\sep/k)$-module 
of $n$-th roots of unity in $k^\sep$. For $d\in k^\times$, we have
\[H^1(k,\mu_n^{(d)})= \frac{\GG_m^{(d)}(k)}{\GG_m^{(d)}(k)^n}.\]
\end{lemma}

\begin{lemma}\label{L:casselsmap}
Let $\tphi\colon \tJ\to\J$ be a polarized isogeny between principally 
polarized abelian surfaces with kernel 
$\krn^{(-3d)}=\mu_3^{(d)}\times\mu_3^{(d)}$. Suppose that $\J$ is 
the Jacobian of a genus $2$ curve of the form
 \[\begin{aligned}\C\colon y^2&=-3d\big(G_1(x)^2+ \lambda_1 H_1(x)^3\big)\\
                               &=-3d\big(G_2(x)^2+ \lambda_2 H_2(x)^3\big),
\end{aligned}\]
where the $3$-torsion subgroup with generators supported at $H_1(x)=0$ 
and at $H_2(x)=0$ is the kernel of an isogeny $\phi\colon \J\to \tJ$ 
such that $\phi\circ\tphi=3$.
Then the connecting homomorphism
\[\tgamma\colon \frac{\tJ(k)}{\phi J(k)}\to H^1(k,\krn^{(-3d)})=
\frac{\GG_m^{(d)}(k)}{\big(\GG_m^{(d)}(k)\big)^3}
\times\frac{\GG_m^{(d)}(k)}{\big(\GG_m^{(d)}(k)\big)^3}\]
is induced by the partial map
\[\begin{matrix}
\C&\dashrightarrow&\GG_m^{(d)}&\times&\GG_m^{(d)}\\[5pt]
(x,y)&\longmapsto& \Big( y-\sqrt{d}G_1(x)&,&y-\sqrt{d}G_2(x) \Big)
\end{matrix}\]
\end{lemma}

\begin{proof}This is a direct application of the theory developed 
in \cite{schaefer:sel} and \cite{BPS:quartic}.
\end{proof}

From Theorems~\ref{T:rst_family} and \ref{T:Ctilde} we can obtain 
isogenies $\phi\colon \J\to\tJ$ and $\tphi\colon \tJ\to \J$ with 
kernels $\krn=\ZZ/3\times\ZZ/3$ and $\krn^\vee=\mu_3\times\mu_3$ 
respectively. We use Lemma~\ref{L:casselsmap} to compute 
$\Sel^{\phi}(\J/\QQ)$ and $\Sel^{\tphi}(\tJ/\QQ)$. As suggested by 
the lemma, we represent the cohomology classes by elements of 
$\Q(\sqrt{-3})^\times/\Q(\sqrt{-3})^{\times3}$ and 
$\Q(\sqrt{-3})^\times/\Q(\sqrt{-3})^{\times3}$ respectively.

We take $S$ to be the set of primes consisting of $3$ and the primes 
of bad reduction of $\C$. By \cite{BPS:quartic}*{Proposition~9.2}, the Selmer groups 
lie in the subgroups that are unramified outside $S$. We can represent 
those using $S$-units in $\Q(\sqrt{-3})$ and $\Q$. This already provides 
us with explicit finite groups that contain the Selmer groups. The 
remaining conditions come from the local images at $v\in S$ (note 
that $\RR^\times/\RR^{\times 3}$ is trivial, so the archimedian place 
does not provide any information). With the explicit description of 
the maps $\gamma_v$ and $\tgamma_v$ we can generate elements in their 
images. Using \cite{schaefer:sel}*{Lemma~3.8, Proposition~3.9} and some basic
basic diagram chasing we have
\begin{equation}\label{eq:localorder}
\# \frac{\tJ(\Q_v)}{\phi\J(\Q_v)}\,
\# \frac{\J(\Q_v)}{\tphi\tJ(\Q_v)}
=  \frac{\# \krn(\Q_v) \,\# \krn^\vee(\Q_v)}{ |3|_p^2}=
\begin{cases}
  81&\text{if }v=3\text{ or }v\equiv 1\pmod{3}\\
   9&\text{if }v\equiv 2\pmod{3}
\end{cases},
\end{equation}
so we know when we have found enough elements to generate the entire 
image. By explicitly computing the restriction maps 
$\QQ^\times/\QQ^{\times 3}\to \QQ_v^\times/\QQ_v^{\times 3}$ and similarly 
for $\QQ(\sqrt{-3})$, we can compute the Selmer groups using essentially 
the definition in~\eqref{E:selmerdef}.

We now apply the above theory to several examples,
which combine the information of standard $2$-descent
and the above descent via $\krn$-isogeny and $\krn^\vee$-isogeny.
In particular, this will give the first examples of
nontrivial 3-part of the Tate-Shafarevich group on a
non-reducible abelian surface.

The first example illustrates a situation where a $3$-isogeny descent 
can be used to obtain a sharper rank bound than one can get from a 
$2$-descent and hence exhibit some non-trivial $2$-torsion elements 
in the Tate-Shafarevich group. There are alternative methods to do 
this, which show in the process that there is no $4$-torsion, but 
obtaining unconditional results through these is too computationally 
expensive at present. The computations involved in these examples are easily
reproduced using Magma \cite{magma} and the software we have made available at
\cite{bft:electronic}.

\begin{example}\label{E:-3,-3,-3}
Let $\J$ be the Jacobian of the curve
\[\sC_{-3,-3,-3}\colon y^2=(12x^3-105)^2-12(x^2 - 3x - 3)^3.\]
Then $\J(\QQ)\simeq(\ZZ/3)^2\times\ZZ$ and $\Sha(J/\QQ)[6]\simeq(\ZZ/2)^3$.

Furthermore, conditional on the Generalized Riemann Hypothesis for a certain degree $12$ number field, we have $\Sha(J/\QQ)[6^\infty]=(\ZZ/2)^3$.
\end{example}

\begin{proof}
Using the isogenies $\phi\colon \J\to\tJ$ and $\tphi\colon \tJ\to \J$ we find
\[\Sel^{\tphi}(\tJ/\QQ)=(\ZZ/3)^4\text{ and }\Sel^\phi(\J/\QQ)=0.\]
We already know that $\J[3](\QQ)$ is non-trivial and further computation 
shows that $\J(\QQ)^{\operatorname{tors}}=(\ZZ/3)^2$. 
A height computation shows that divisors supported at $101x^2+21x+147=0$ 
and at $x^2-6x-45=0$ generate independent classes in $\J(\QQ)$. From this 
one can deduce the structure of $\J(\QQ)$ and that $\Sha(\J/\QQ)[\phi]$ 
and $\Sha(\tJ/\QQ)[\tphi]$ are trivial. From $3=\tilde{\phi}\circ\phi$ 
it follows that $\Sha(J/\QQ)[3]$ is trivial as well.

A $2$-descent on $\J$ yields $\Sel^2(\J/\QQ)=(\ZZ/2)^5$, which shows 
that $\Sha(\J/\QQ)[2]=(\ZZ/2)^3$. Indeed $\J$ is {\em{odd}} in the sense 
of \cite{poosto:casselstate}.

In order to prove there is no $4$-torsion in $\Sha(\J/\QQ)$ we observe 
that $\J^{(-2)}(\QQ)$ has rank at least $3$, which can be shown by 
exhibiting enough points and a height pairing computation. One can 
compute that $\Sel^2(\J/\QQ(\sqrt{-2}))=(\ZZ/2)^5$, provided one 
verifies a certain class group computation for which one presently requires 
the Generalized Riemann Hypothesis. Since the rank of $\J(\QQ(\sqrt{-2}))$ 
is the sum of the ranks of $\J(\QQ)$ and $\J^{(-2)}(\QQ)$, one concludes 
that $\Sha(\J/\QQ(\sqrt{-2}))[2]$ is trivial. Since the restriction 
map $\Sha(\J/\QQ)\to\Sha(\J/\QQ(\sqrt{-2}))$ can only kill elements of 
order $2$, the statement follows.

It is worth noting that $\Sel^2(\J/\QQ)=(\ZZ/2)^5$, so 
$\Sha(J^{(-2)}/\QQ)[2^\infty]=(\ZZ/2)^2$.

In fact, using a visibility argument \cite{bruinflynn:vis} we find 
that $\J^{(2)}=\ZZ^3$ and that $\Sel^2(\J/\QQ(\sqrt{2}))=(\ZZ/3)^2$. 
It follows that $\Sha(\J/\QQ(\sqrt{2}))[2]=0$. Since the restriction 
map $\Sha(\J/QQ)\to\Sha(\J/\QQ(\sqrt{2}))$ can only kill $2$-torsion, 
the statement in the example follows.
\end{proof}

\begin{example}\label{E:-2,1,2}
Let $\tJ$ be the Jacobian of the curve
\[\tilde{\sC}_{-2,1,2}\colon 
y^2=-48(83x^3 + 498x^2 - 996x + 581)^2-3984(15x^2 - 26x + 10)^3.\]
Then $\tJ(\QQ)\simeq\ZZ$ and $\Sha(\tJ/\QQ)[3]\simeq(\ZZ/3)^2$.
\end{example}

\begin{proof}
We find
\[\Sel^{\tphi}(\tJ/\QQ)=(\ZZ/3)^5 \text{ and } \Sel^\phi(\J/\QQ)=0.\]
With $\J(\QQ)[3]=(\ZZ/3)^2$ and $\tJ(\QQ)[3]=0$, this implies that 
$\tJ(\QQ)$ is of rank at most $3$.

From a $2$-descent we find $\Sel^2(\J/\QQ)=(\ZZ/2)$. Furthermore, we 
find a non-torsion point in $\J(\QQ)$, so we find that 
$J(\QQ)=(\ZZ/3)^2\times\ZZ$ and $\tJ(\QQ)=\ZZ$. Combined with the 
result above, this yields that $\Sha(\tJ/\QQ)[3]=(\ZZ/3)^2$.
\end{proof}

\begin{example}\label{E:2,-1,-2}
The Jacobian $\tJ$ of the curve
\[\tilde{\sC}_{2,-1,-2}\colon
y^2=-48(706x^3 + 2118x^2 + 4236x + 353)^2+16944(5x^2 - 14x - 30)^3
\]
has $6$-torsion in $\Sha(\tJ)$ and $\tJ(\QQ)=\{0\}$.
\end{example}
\begin{proof}
Let $\J$ be the Jacobian of
\[\sC_{2,-1,-2}\colon 
y^2=12x^6 + 72x^5 + 312x^4 + 688x^3 + 768x^2 + 192x + 68.\]
A direct computation shows that $\Sel^{\tphi}(\tJ/\QQ)=(\ZZ/3)^4$ and 
$\Sel^{\phi}(\J/\QQ)=0$. The torsion $\J[3](\QQ)=(\ZZ/3)^2$ explains 
two factors, so either $\tJ(\QQ)$ is of rank $2$ or $\Sha(\tJ/\QQ)[3]$ 
is non-trivial.

Similarly, a $2$-descent shows that 
$\Sel^2(\tJ/\QQ)\simeq\Sel^2(\J/\QQ)=(\ZZ/2)^2$. Further computation shows 
that $\Sel^2(\J^{(3)}/\QQ)=(\ZZ/2)^3$ and that $\J^{(3)}(\QQ)\simeq\ZZ^3$.
Further computation shows that $\Sel^2(J/\QQ(\sqrt{3}))=(\ZZ/2)^3$ as well, 
so $\J^{(3)}(\QQ(\sqrt{3}))\simeq\ZZ^3$ as well. It follows that 
$\J(\QQ)\simeq (\ZZ/3)^2$ and that 
$\Sha(\J/\QQ)[2^\infty]=\Sha(\tJ/\QQ)[2^\infty]=(\ZZ/2)^2$.
\end{proof}


\end{document}